\documentclass[11pt]{amsart}

\usepackage{amsmath, amsfonts, amssymb, setspace, xypic}
\input xy
\xyoption{all}
\theoremstyle{plain}
\newtheorem{theorem}{Theorem}[section]
\newtheorem{lemma}[theorem]{Lemma}
\newtheorem{lemma-def}[theorem]{Lemma-Definition}
\newtheorem{proposition}[theorem]{Proposition}
\newtheorem{corol}[theorem]{Corollary}

\theoremstyle{definition}
\newtheorem{definition}[theorem]{Definition}

\newcommand{\Q}{\mathbb{Q}}

\newcommand{\F}{\mathbb{F}}
\newcommand{\FF}{\F}
\newcommand{\Z}{\mathbb{Z}}
\newcommand{\ZZ}{\Z}
\newcommand{\N}{\mathbb{N}}

\newcommand{\PP}{\mathbb{P}}
\newcommand{\Lm}{\mathcal{L}}

\newcommand{\ord}{\mathrm{ord}\,}

\newcommand{\Lring}{\Lm_{\text{ring}}}

\newcommand{\Laff}{\Lm_{\text{aff}}}
\newcommand{\Laffinf}{\Laff^{\pi}}
\newcommand{\LaffinfQE}{\widehat{\Laffinf}}

\newcommand{\acm}{\mathrm{ac}_{\pi^m}\,}
\newcommand{\mathand}{\quad \text{and} \quad}

\newcommand{\rv}{\mathrm{rv}}

\title{Cell decomposition for semi-affine structures on $p$-adic fields}

\author{Eva Leenknegt}
\address{Department of mathematics \\ Purdue University \\ \newline150 N. University Street \\ West- Lafayette, IN 47907 \\ USA}
\email{eleenkne@math.purdue.edu}
\urladdr{http://www.math.purdue.edu/~eleenkne}

\date{}
\begin{document}

\begin{abstract}
We use cell decomposition techniques to study additive reducts of $p$-adic fields. We consider a very general class of fields, including fields with infinite residue fields, which we study using a multi-sorted language. The results are used to obtain cell decomposition results for the case of finite residue fields. We do not require fields to be Henselian, and we allow them to be of any characteristic.\end{abstract}
\maketitle
\section{Introduction}
It is hard to overstate the importance of cell decomposition techniques for the study of $o$-minimal structures. The technique made it possible to obtain results for a wide array of topics, ranging from the study of definable invariants to differentiability of definable functions, see eg. van den Dries \cite{vdd-98} for details. 

Another example is the classification of reducts of $(R,+,\cdot, <)$ by Peterzil \cite{mpp-92,pet-93,pet-92} and others. One of the most striking results he obtained  is the fact that there exists only a single structure between the structure of semi-algebraic sets and the semi-linear sets of $(R,+,\{\lambda_a\}_{a \in R})$: a structure where multiplication is definable only on a bounded interval.
The question whether a similar result would exist in the $p$-adic context was one of the motivations for this paper: a good understanding of semi-affine structures is a necessary first step
towards answering this question. 
In the upcoming papers \cite{lee-2012.1, lee-2012.2} we will report our findings.
 
For $p$-adic structures, a number of cell decomposition results do exist. Probably the most well-known is the cell decomposition result for semi-algebraic sets by Denef \cite{denef-86}, which allowed him to give a new proof of Macintyre's quantifier elimination result \cite{mac-76}, and which has been a very useful tool in the study of $p$-adic integrals , see eg Denef \cite{denef-86} or Cluckers and the author \cite{clu-lee-2008}. 
Haskell and Macpherson \cite{has-mac-97} developed $P$-minimality as a $p$-adic alternative to $o$-minimality, to study (expansions) of  $p$-adically closed fields. It was shown later by Mourgues \cite{mou-09} that such structures admit cell decomposition (using Mourgues definition of cells) if and only if they have definable Skolem functions.

Most existing $p$-adic cell decomposition results focus on (expansions of) the semi-algebraic structure. This poses a complication for obtaining $p$-adic equivalents of Peterzil's result, because  there does not really exist a minimality theory for weak $p$-adic structures.
In a previous paper \cite{clu-lee-2011} we proposed to consider all structures $(K,\Lm)$, where $K$ is a $p$-adic field and the $\Lm$-definable subsets of $K$ are the same as the $\Lring$-definable subsets of $K$. This is a direct $p$-adic equivalent of $o$-minimal reducts of $(R,+,\cdot, <)$. Unfortunately, we were unable to prove whether such structures would always admit some form of cell decomposition. We gave a few suggestions in \cite{lee-2011b}, but it seems to be rather difficult even to suggest a useful general notion of cells, so a general cell decomposition theorem for such structures still seems inaccessible.

A natural first step is to study the properties of individual structures. In \cite{lee-2011b} we consider some very weak structures (where even addition is not definable everywhere), and in this paper we will look at the $p$-adic equivalent of semi-linear sets. Some time ago, Liu \cite{liu-94}  obtained a cell decomposition for the semi-linear structure $(\Q_p; +,-, \{\overline{c}\}_{c\in \Q_p}, \{P_n\}_n)$,  where $\overline{c}$ is a symbol for scalar multiplication $x \mapsto cx$, and the $P_n$ are the nonzero $n$-th powers. 
This paper describes similar structures, but in a more general context.

We will consider structures $(K,\Lm)$, where $\Lm$ is a semi-affine language and where $K$ is a $\Z$-field: a valued field that satisfies the following extra conditions. Write  $\Gamma_K$ for the value group,  and let  $R_K$ be the valuation ring of $K$.
\begin{definition}
A $\Z$-field is a valued field $K$ that contains an element $\pi$ of minimal positive valuation.  Further, we require that $\Gamma_K$ is a $\Z$-group (that is, $\Gamma_K/\Z$ is divisible) and that there exist angular component maps $\acm: K\to R_K/\pi^mR_K$.
\end{definition}
We will assume that the valuation is normalized such that $\ord  \pi =1$. 
The required
angular component maps always exist  if $\Gamma_K = \Z$ and $K$ has a uniformizing element $\pi$. The proof is similar to the proof of Lemma  \ref{lemma:acm}.
Note that we do \emph{not} put any conditions on the residue field $\F_K$ and the characteristic of $K$. Moreover, we do not require the field $K$ to be Henselian.

When studying structures on valued fields, often multi-sorted languages are considered, typically consisting of a field sort and various other, auxiliary sorts used to encode information concerning the residue field and angular components.   See for example Pas \cite{pas-89,pas-90}, who used multi-sorted languages to study semi-algebraic structures for fields with infinite residue fields. This approach was extended to fields with analytic structure by Cluckers, Lipshitz and Robinson \cite{clr-06}.  Other recent examples include Scanlon, who used a multi-sorted language to study valued $D$-fields \cite{sca-03},  and  Cluckers and Loeser \cite{clu-loe-07}, who obtain cell decomposition for henselian valued fields of characteristic zero.

Most of the examples given above are essentially multi-sorted versions of (extensions of) the language of valued fields. We present a multi-sorted  language where full multiplication is not definable,  but such that `multiplicative' relations like the valuation of $x$ modulo $n$ are still definable.  (This relation is equivalent to $x$ being in certain cosets of the set of $n$th powers.)
 The valued field $K$ will be the main sort, equipped with the language $(+, \cdot_{\pi}, |)$.
The function $\cdot_{\pi}$ is defined as \[\cdot_{\pi}:K\to K: x \mapsto \pi x.\]
The divisibility relation $|$ is defined as  $x \mid y$ iff $\ord x \leqslant \ord y$.
 
The auxiliary sorts $\Lambda_{n,m}$ are constructed as follows.
Since $\Gamma_K$ is a $\Z$-group, there exist maps
\(\gamma_n: K^{\times} \to \{0,\, \ldots,\, n-1\},\)
where  $\gamma_n(x)$ is the remainder of $\ord x$ after division by $n$. 
 For every $x \in K^{\times}$, put 
\[\rho_{n,m}(x) = \pi^{\gamma_n(x)}\acm(x).\]
Extend this to $K$ by putting $\rho_{n,m}(0) = 0$. 
Our auxiliary sorts will then be the sets of equivalence classes:
%
\[\Lambda_{n,m}:=\{\rho_{n,m}(x)\ | \ x \in K\}.\]

 The maps $\rho_{n,m}$ project the main sort $K$ onto the auxiliary sorts $\Lambda_{n,m}$.  The language on the auxiliary sorts contains no symbols. Schematically, this gives us the following  language $\Laff^{\pi}$:
\[\xymatrix{\ar[d]^{\rho_{n,m}}K& \hspace{-50pt}(+, \cdot_{\pi}, | ) \\
\{\Lambda_{n,m}\}_{n,m}&\hspace{-25pt}}\]

Note that this language does not use the value group as a seperate sort. However, the sets $\Lambda_{n,m}$ retain information on the value group, modulo an integer $n$.
Let us give some examples of relations that are definable in this language.

\begin{lemma}\label{lemma:defvb} Let $K$ be a $\Z$-field.
For every $k, n \in \N$, the following subsets of $K$ are  $\Laff^{\pi}$-definable:
\begin{enumerate}
\item $\{ (x,y) \in K^2 \mid \ord x = \ord y \}$
\item $\{ x\in K \mid \ord x \equiv k \mod n \}$
\end{enumerate}
\end{lemma}
\begin{proof}
The relation $\ord x = \ord y$ is equivalent with $x \mid y \wedge \neg (\pi x \mid y)$. We can use this to express that $\ord x =  \gamma$ for any $\gamma \in \gamma_K$ by substituting $y$ for a suitable constant from $K$. For $\lambda \in \Lambda_{n,m}$ we can then express that $\ord \lambda \equiv k \mod n$ in the following way:
\[ \ord \lambda \equiv k \mod n \leftrightarrow (\exists x \in K)[\rho_{n,m}(x) = \lambda \text{ \ and \ } \ord x = k].\]
We can now use the formula 
\[(\exists \lambda \in \Lambda_{n,m})[(\ord \lambda \equiv k \hspace{-2pt}\mod n) \text{\ \ and \ } \rho_{n,m}(x) = \lambda].\]
to define the set consisting of all $x \in K$ such that $\ord x \equiv k \mod n$. 
\end{proof}
 If $K$ is a $\Z$-field with infinite residuefield, 
 the set $\{x \in K \mid \ord x \equiv k \mod n\}$ cannot be defined in $\Laff^{\pi}$ without using $K$-quantifiers. 
 To remedy this, 
we  expand the language in Section \ref{subsec:langdef}, thus obtaining an additive variant of the language studied by Pas. 
In Section \ref{subsec:celdec} and \ref{subsec:defsetfun},
 we show that  $\Z$-fields
 admit elimination of $K$-quantifiers in this extended language $\LaffinfQE$. The proof  uses  cell decomposition techniques. We also give a characterization of the definable functions $f: K^n \to K^m$. 

In Section \ref{sec:finres}, we restrict our attention to fields with finite residue field. For such fields, we can `collapse' the multisorted language to a language with just one sort, and derive cell decomposition and quantifier elimination from the results we obtained for the multisorted language. 

To make the distinction between mono- and multisorted languages clear, we will use the following terminology. The definable sets of our multi-sorted language are called `semi-additive' sets. We will refer to the mono-sorted languages we deduce from this as  `semi-affine' languages. `Semi-linear' sets are the definable sets of the structure Liu studied on $\Q_p$. We will compare our results for semi-affine sets with  existing results for semi-linear and semi-algebraic sets. In particular, we give a characterization of definable functions in Section \ref{subsec:skol}, and in Section \ref{subsec:clas}, we give some examples to show that classification by definable bijection is not quite as simple as it is for semi-algebraic sets. (It was shown by Cluckers \cite{clu-2000} that any two infinite $p$-adic semi-algebraic sets are isomorphic if and only if they have the same dimension.)
 
\section{Affine structures with infinite residue field}

\subsection{Definition of the languages $\Laffinf$ and $\LaffinfQE$} \label{subsec:langdef}
Let $K$ be a valued field with value group $\Gamma_K$ and valuation ring $R_K$. Let $\pi$ be an element $\pi$ of minimal positive valuation, such that $\ord \pi =1$. We use the notation $\acm$ for the angular component maps $\acm: K\to R_K/\pi^mR_K$.

The only symbol for multiplication we included in $\Laff^{\pi}$ is $\cdot_{\pi}$. However, as addition is definable,  scalar multiplication by every $n \in \N$ is definable. This implies that if $K$ has characteristic zero, multiplication by every $c \in \Q(\pi)$ is definable. If char($K$)=$p$, we can define scalar multiplication for every $c \in \F_p(\pi)$. \\In general, if we denote the prime field of $K$ by $\PP_K$, we can thus define a scalar multiplication map $\overline{c} : K\to K : x\mapsto cx$ for every $c \in \PP_K$.

We added the symbol $\cdot_{\pi}$ because of the functions it induces on the auxiliary sets. We do not  include symbols for scalar multiplication by other constants, as we want to keep the language as basic as possible. However, it is possible to define variations on $\Laff^{\pi}$ that contain a wider range of symbols for scalar multiplication. It is easy to see that such languages can be studied in a similar way as $\Laff^{\pi}$. In fact, we refer to these  related languages when we consider the case of finite residue fields.

The addition map for the main sort $K$ induces addition functions $+_r^{(n,m)}$ on the auxiliary sorts $\Lambda_{n,m}$,
%
where $r\in \N$ is such that $rn<m$. 
If $rn <=\ord \frac{y}{x}<(r+1)n$ (and some additional conditions if $r=0$), these functions are designed to satisfy the relation
\[\rho_{n,m}(x+y) = \rho_{n,m}(x) +_r^{n,m} \rho_{n,m}(y).\] 

Why do we need to consider multiple addition functions on the auxiliary sorts? To see this, let us compare with a similar construction in a different language.
In \cite{fle-2011}, Flenner considers a language with auxiliary sorts \[RV_{\gamma} := K^{\times} \backslash (1+ M_{\gamma}),\] where $M_{\gamma} = \{ x\in R_K \mid \ord x > \gamma\}$. The quotient map, which is denoted $\rv_{\gamma}:= K^{\times} \mapsto RV_{\gamma}$, induces  an addition function $\oplus_{\gamma}$ on each sort $RV_{\gamma}$, that is compatible with the addition in the main sort, in the sense that
\[\rv_{\gamma}(x+y) = \rv_{\gamma}(x) \oplus_{\gamma} \rv_{\gamma}(y),\] for all $x,y$ for which $\ord (x+y) = \min\{\ord x, \ord y \}.$

If this condition is not satisfied, the operation $\oplus$ is not welldefined, since then $\rv_{\gamma}(x+y)$ depends on the representatives $x$ and $y$, and not only on $\rv_{\gamma}(x)$ and $\rv_{\gamma}(y)$. To define the value of  $\lambda_1 \oplus_{\gamma} \lambda_2$ for $\lambda_1, \lambda_2 \in RV_{\gamma}$, 
one chooses representatives $x_i$ such that $\rv_{\gamma}(x_i) = \lambda_i$, and then puts $\lambda_1 \oplus_{\gamma} \lambda_2 := \rv_{\gamma}(x_1 + x_2).$ If $\ord (x_1+x_2) = \min\{\ord x_1, \ord x_2 \},$ this value does not depend on the chosen representatives, so this addition is well-defined. 

 The main difference between the sorts $RV_{\gamma}$ and the sorts $\Lambda_{n,m}$ is that $\rv_{\gamma}(x)$ remembers the order of $x$, while $\rho_{n,m}(x)$ only retains the order modulo $n$. 
 Hence we will have to be more careful, since every equivalence class in $\Lambda_{n,m}$ contains representatives with different orders. 
 Let $\lambda, \lambda' \in \Lambda_{n,m}$ and suppose that we want to define $\lambda\oplus \lambda'$. The outcome will depend on the distance of the chosen representatives, by the followning lemma:
  \begin{lemma}\label{lemma:rhoab}
Put $\delta \in \{-1,1\}$. Suppose that $\rho_{n,m}(a) = \lambda$ and $\rho_{n,m}(b) = \mu$, then $\rho_{n,m}(a+\delta b)$ equals
\[ \left\{\begin{array}{lcl}
\hspace{-4pt}\lambda &\text{if}& m + \ord a \leqslant \ord b, \\
\hspace{-4pt} \rho_{n,m}(\lambda +\delta \mu \pi^{rn}), \text{with}  \ rn = \ord (\frac{\mu a}{\lambda b})& \text{if} & -m +\ord b <\ord a <\ord b,\\
\hspace{-4pt} \rho_{n,m}(\lambda + \delta \mu)&\text {if}& \ord a = \ord b = \ord (a + \delta b).
\end{array}\right.\]
\end{lemma}
\begin{proof}Left as an exercise.
\end{proof}
If we want to define addition maps, we will have to take these different possibilities into account. Not that this is a bad thing: this means that we can use the auxiliary sorts $\Lambda_{n,m}$ to encode information about the distance between elements of $K$. This will be important when we want to achieve cell decomposition. 

Let us now give a precise definition of the addition maps $+^r_{n,m}$.
 If $r \geqslant 1$, let $\lambda +^{(n,m)}_r \lambda'$  be the unique value $\rho \in \Lambda_{n,m}$ such that 
\[(\exists x,y \in K)\left[\begin{array}{cl}&[\rho_{n,m}(x)=\lambda] \wedge [\rho_{n,m}(y)= \lambda'] \\  \wedge & [0 \leqslant \ord x <n]
\wedge [0 \leqslant \ord y <n] \\ \wedge & \rho_{n,m}(x+\pi^{rn}y) = \rho \end{array}\right]\] 
(The above formula cannot be used if $\lambda=0$ or $\lambda'=0$. We can extend the definition to these cases by putting $0 +^{(n,m)}_r \lambda = \lambda +^{(n,m)}_r 0 =\lambda$.)

 If $r=0$, the operation above does not always yield a unique result. For this reason, we will restrict the domain to
$D_+ \hspace{-2pt}:=\hspace{-1pt} \{(\lambda, \lambda') \in \Lambda_{n,m}^2 \ | \ \phi(\lambda,\lambda')]\}$,
where $\phi(\lambda,\lambda')$ is  the formula
\[(\forall x, y \in K)\left[[\rho_{n,m}(x) = \lambda \wedge \rho_{n,m}(y) = \lambda' \ \wedge \ord x =\ord y]\Rightarrow\ord (x+y) = \ord x\right].\]
For $(\lambda, \lambda') \in D_+$, define $\lambda +^{(n,m)}_0 \lambda'$ by the same formula as for $r\geqslant 1$; \ put   $\lambda +^{(n,m)}_0 \lambda':=0$ if $(\lambda,\lambda') \notin D_+$.
Analogously, we can define functions $ -^{(n,m)}_r $. If the domain is clear from the context, we will simply write $+_r$ or $-_r$.


We are now ready to introduce the language $\LaffinfQE$, which is a definitional expansion of $\Laffinf$, obtained by adding symbols for the functions we discussed above, and symbols for the relation
\[\equiv_{n,k}(\lambda) \leftrightarrow \ord \lambda \equiv k \mod n,\]
which we showed to be definable in the proof of Lemma \ref{lemma:defvb}.
Schematically, this gives us the following language:
\[\xymatrix{\ar[d]^{\rho_{n,m}}K& \hspace{-160pt}(+, \overline{c}_{c\in \PP_K(\pi)}, | ) \\
\{\Lambda_{n,m}\}_{n,m}&\hspace{-20pt}
(\{+_r^{(n,m)}\}_{r \in \N},\{-_r^{(n,m)}\}_{r\in \N},\{\equiv_{n,k}\}_{k\in\N})}\]
We will show that $\Z$-fields admit quantifier elimination and cell decomposition in this language.

Remark: The same notation $\rho_{n,m}$ will also be used for the natural projection maps
\[\rho_{n,m}:\Lambda_{kn,m'}\to \Lambda_{n,m},\]
with $k \in \N\backslash\{0\}$, $m' \geqslant m$. These maps are clearly definable in our original language. We assume that our extended language contains symbols for these maps. These projection maps are `compatible' with the functions we defined on the $\Lambda_{n,m}$: for example for the addition maps we have
\[\rho_{n,m}(\lambda +_r^{(kn,m')} \lambda') = \rho_{n,m}(\lambda) +_r^{(n,m)} \rho_{n,m}(\lambda').\] 

\subsection{Subsets of $K^k$ definable without $K$-quantifiers in $\LaffinfQE$} \label{subsec:form_of_kqf-formula}
In this section we will give a short description of $K$-quantifier-free definable subsets of $K^k$. Let $\phi(x)$ be a formula without $K$-quantifiers, and such that all free variables $x= (x_1, \ldots, x_k)$ are $K$-variables. We use the following notation.
\begin{itemize}
\item Let $g_{i,n,m}(\lambda_1, \ldots, \lambda_r)$ denote a term in the $\Lambda_{n,m}$-sort. 
\item Let $f_i(x)$ denote a linear polynomial in the $K$-variables $x$ with coefficients in $\PP_K(\pi)$ and constant term in $K$. We call this a $(\PP_K(\pi),K)$-linear polynomial.
\item Let $\theta_{i,k',n,m}(\lambda_1, \ldots, \lambda_{k'})$ be a formula in the $\Lambda_{n,m}$-sort with $k'$ free variables.
\end{itemize}
With this notation, $\phi(x)$ is a boolean combination of expressions $\phi_{i,n,m}(x)$ and $\psi_{i,j}(x)$:
\begin{enumerate}
\item[(a)] Put $\mu_{j}(x) := g_{j,n,m}(\rho_{n,m}(f_{j,1}(x)), \ldots, \rho_{n,m}(f_{j,r}(x)))$, then $\phi_{i,n,m}(x)$
 is a formula of the form $\phi_{i,n,m}(x) \leftrightarrow \theta_{i,k',n,m} (\mu_1(x),\ldots, \mu_{k'}(x)))$.
\item[(b)] $\psi_{i,j}(x) \leftrightarrow \ord f_i(x) \ \square \ \ord f_j(x)$, where $\square$ may denote $<,\leqslant, >,\geqslant,=$.
\end{enumerate}
Note that we may assume that the same value of $n$ and $m$ occurs in every expression of type $\phi_{i,n,m}$. (Indeed, expressions $\phi_{i,n,m}$ and $\phi_{j,n',m'}$, can (with the help of projection maps $\rho_{n,m}$ and $\rho_{n',m'}$) be rewritten to expressions $\phi_{i,n'',m''}, \phi_{j,n'',m''}$, where $n'' = \text{lcm}\{n,n'\}$, and $m'' = \max\{m,m'\}$.) \\
Also, since any negation of an expression of type (a) or (b) can again be rewritten as an expression of the same form, $\phi(x)$ can be obtained by taking (a finite number of) conjunctions and disjunctions of such expressions. 
Furthermore, note that any expression of type (a) is equivalent with
\[(\exists \lambda_{ij} \in \Lambda_{n,m})\left[\left(\bigwedge_{i,j} \rho_{n,m}(f_{i,j}(x)) = \lambda_{i,j}\right) \wedge \psi(\lambda_{11}, \ldots, \lambda_{k'r})\right]\]
where the formula $\psi$ is defined as \[\psi(\lambda_{11}, \ldots, \lambda_{k'r}) \leftrightarrow \theta_{i,k',n,m}(g_{1,n,m}(\lambda_{11}, \ldots, \lambda_{1r}), \ldots, g_{k',n,m}(\lambda_{k'1}, \ldots, \lambda_{k'r})).\]
 It follows then immediately that $\phi(x)$ is in fact a disjunction of expressions of the form
\begin{equation} (\exists \lambda \in \Lambda_{n,m}^r)\left[\phi_1(x) \wedge \left(\bigwedge_i \rho_{n,m}(f_i(x)) = \lambda_i\right) \wedge \phi_2(\lambda) \right]\label{eq:Kqf-definable}\end{equation}where $\lambda = (\lambda_1, \ldots, \lambda_r)$. Here $\psi_1$ is a quantifier-free formula in the language of the main sort $K$, and $\phi_2$ is a formula in the language of the  $\Lambda_{n,m}$-sort (not necessarily quantifierfree).\\\\
\subsection{Cell Decomposition} \label{subsec:celdec}
The following notation will be convenient.
Let $D \subseteq \Lambda_{n,m}^r \times K^k$ be a definable set, and suppose that
 $r'\leqslant r$, $k'\leqslant k$ and $k'+r' <k+r$. For any $(\rho, b) \in \Lambda_{n,m}^{r'}\times K^{k'}$, the notation $D(\rho,b)$ denotes the set
\[D(\rho,b):=\{(\lambda, x) \in \Lambda_{n,m}^{r-r'}\times K^{k-k'} \ | \ (\rho, \lambda, b,x) \in D\}. \]
We next define our notion of cells. This notion of cells is closely analogous to the notions of cells used for other multi-sorted languages. 
\begin{definition}
A cell in $\Lambda_{n,m}^r \times K^{k+1}$ is a set 
\[\left\{(\lambda,x,t)\in D_{n,m} \times D_K\times K \ \left| \ \begin{array}{l} \ord a_1(x) \ \square_1 \ \ord (t-c(x)) \ \square_2 \ \ord a_2(x),\\\text{and }\rho_{n,m}(t-c(x)) \in D(\lambda, x)\end{array}\hspace{-2pt}\right\}\right.\hspace{-2pt}\]
where
\begin{itemize}
\item $D_{n,m}$ is a subset of $\Lambda_{n,m}^r$, $\LaffinfQE$ -definable without $K$-quantifiers,
\item $D_K$ is a subset of $K^k$, $\LaffinfQE$ -definable without $K$-quantifiers,
\item $D$ is a  subset of $\Lambda_{n,m}^{r+1}\times K^k$, $\LaffinfQE$ -definable without $K$-quantifiers,
\item the functions $a_i(x), c(x)$ are $(\PP_K(\pi),K)$-linear polynomials in the variables $(x_1,\ldots, x_k)$. We call $c(x)$ a center of the cell,
\item $\square_i$ may denote either $<$ or `no condition'.
\end{itemize}
\end{definition}
Note that in the description of such a cell, $\square_i$ can only denote a strict inequality `$<$'. However, in the expressions in $(b)$ of Subsection \ref{subsec:form_of_kqf-formula}, we also used `$\leqslant$' and `$=$'.  We can exclude these options since they can be expressed in terms of a strict inequality. Indeed, \[\ord f(x) \leqslant \ord g(x) \Leftrightarrow \ord f(x) < \ord \pi g(x),\] and \[\ord f(x) = \ord g(x) \Leftrightarrow \ord f(x) < \ord \pi g(x) < \ord \pi^2 f(x).\]
As a first step, we show that cells behave well when taking finite intersections.

\begin{proposition}\label{prop:intersection}
Let $C_1, C_2$ be two cells in $\Lambda_{n,m}^r \times K^{k+1}$. The intersection $C_1 \cap C_2$ can be partitioned as a finite union of cells.
\end{proposition}
\begin{proof}
First consider semi-cells of the following form:
\[C_{c}^{D}(a_1,a_2)\hspace{-1pt} :=\hspace{-1pt}\left\{(\lambda, x,t) \in D\times K \ | \  \ord a_1(x)\,  \square_1 \, \ord
(t-c(x))\,  \square_2 \, \ord a_2(x)\right\}\hspace{-2pt}, \]
Using the ultrametric property of the valutation, it is easy to see that
the intersection of two semi-cells 
$C_{c_1}^{D_1}(a_1,a_2)$ and $C_{c_2}^{D_2}(b_1,b_2)$ can be
partitioned as a finite union of sets $A$, such that either
 $A$ is the set of all $(\lambda,x,t) \in D\times K$ on which
\begin{equation}
\ord (t-c_1(x)) = \ord (t-c_2(x)) = \ord (c_1(x) -
c_2(x)),\label{vw3}
\end{equation}
with $D$ a subset of $\Lambda_{n,m}^r \times K^k$, definable without $K$-quantifiers,\\
 or $A$ is a semi-cell $C_c^E(e_1,e_2)$, with the center $c(x)$ equal to $c_1(x)$ or $c_2(x)$, such that one of the
following is true on $A$:
\begin{eqnarray}
\ord (t-c(x))& > &\ord(c_1(x)-c_2(x)),\label{vw1}\\
\ord (t-c(x)) & < & \ord (c_1(x)-c_2(x)) .\label{vw2}
\end{eqnarray}
A set that satisfies one of those 3 conditions, say condition $(l)$, will be referred to as a set of type $(l)$.
A general  cell $C_{c}^{D} (a_1,a_2, D_{\rho})$ has
the form:
\[ \left\{(\lambda,x,t) \in D\times K\ \left|  \begin{array}{l}\ord a_1(x)\,  \square_1 \, \ord
(t-c(x))\,  \square_2 \, \ord a_2(x), \\ \text{and }\rho_{n,m}(t-c(x))
\in D_{\rho}(x,\lambda)\end{array}\right\}\right.\] 
We want to intersect two
cells $C_{c_1}^{D_1} (a_1,a_2, D_{\rho}^{(1)})$ and $C_{c_2}^{D_2} (b_1,b_2,
D_{\rho}^{(2)})$. By the discussion above, we can write 
\[C_{c_1}^{D_1} (a_1,a_2, D_{\rho}^{(1)}) \cap C_{c_2}^{D_2} (b_1,b_2, D_{\rho}^{(2)}) = \left( A^{(\ref{vw3})} \cup \bigcup_i A_i^{(\ref{vw1})} \cup \bigcup_j A_j^{(\ref{vw2})}\right)\cap Q\]
where \[Q = \left\{(\lambda,x,t) \in \Lambda_{n,m}^r \times K^{k+1} \ | \begin{array}{l}  \rho_{n,m}(t-c_i(x)) \in D_\rho^{(i)}(x,\lambda), \quad \text{for\ } i = 1,2
 \end{array}\right\}\] and $A_i^{(l)}$ is a set of type $(l)$. We will show that each $A_i^{(l)} \cap Q$ can be written as a finite union of cells .
After a straightforward further partitioning we may suppose that $t-c_1$ and $t-c_2$ are both nonzero, and thus that $0 \not \in D_{\rho}^{(i)}(\lambda,x)$ for any $(\lambda, x) \in D_i$. \\\\The first part of the above intersection is $A^{(\ref{vw3})} \cap Q$. 
If we define $B_1$ to be the set
\[B_1 :=\left\{(\lambda, x, \rho) \in D_1\times \Lambda_{n,m} \ \left| \ \begin{array}{l} \rho \in D_{\rho}^{(1)}(\lambda,x), \\\text{and } \rho +_{0} \rho_{n,m}(c_1(x)-c_2(x)) \in D_{\rho}^{(2)}(\lambda,x)\end{array}\right\}\right.\hspace{-2pt}\]
then $A^{(\ref{vw3})} \cap Q=S$, with \[S:= \left\{(\lambda, x,t) \in (D_1 \cap D_2)\times K \ \left| \begin{array}{l} \ord(t-c_1(x)) = \ord (c_1(x)-c_2(x)), \\\text{and } \rho_{n,m}(t-c_1(x)) \in B_1(\lambda,x)\end{array}\right\}\right.\hspace{-2pt}\]  
Indeed: if $(\lambda,x,t) \in A^{(\ref{vw3})} \cap Q$, then $\rho_{n,m}(t-c_2(x)) = \rho_{n,m}(t-c_1(x)) +_0 \rho_{n,m}(c_1(x)-c_2(x))$, and therefore $\rho_{n,m}(t-c_1(x))\in B_1(\lambda,x)$. On the other hand, the second condition in the description of $S$ implies that $\rho_{n,m}(t-c_1(x)) \,+_0\, \rho_{n,m}(c_1(x)-c_2(x)) \neq 0$, and since $\ord(t-c_1(x)) = \ord(c_1(x)-c_2(x))$, it follows from the definition of $+_0$  that also $\ord(t-c_2(x)) = \ord(t-c_1(x))$. But that means that $\rho_{n,m}(t-c_2(x)) = \rho_{n,m}(t-c_1(x)) +_0 \rho_{n,m}(c_1(x)-c_2(x))$, and thus $\rho_{n,m}(t-c_2(x)) \in D_\rho^{(2)}$, as required.\\\\
On a semi-cell $A_i^{(\ref{vw1})}$ with center $c_1(x)$, the condition    $\ord(t-c_1(x)) > \ord (c_1(x) - c_2(x))$ holds. After a straightforward further partitioning, we get semi-cells $A_{i,j}^{(\ref{vw1})}$ with the same center, such that on each $A_{i,j}^{(\ref{vw1})}$, one of the conditions 
\begin{align}
&\ord (t-c_1(x)) > \ord (c_1(x) - c_2(x)) + m,\text{ \ or \ }\label{easy}\\  
&\ord (t-c_1(x)) = \ord(c_1(x) - c_2(x)) + k, \quad \text{for} \ 0 < k <m\label{harder}
\end{align} holds.
If condition (\ref{easy}) holds on $A_{i,j}^{(\ref{vw1})}$, then we can simply put
\[A_{i,j}^{(\ref{vw1})}\cap Q = \left\{(\lambda, x,t) \in A_{i,j}^{(\ref{vw1})} \ \left| \begin{array}{l} \rho_{n,m}(t-c_1(x)) \in D_\rho^{(1)}(\lambda,x), \\\text{and } \rho_{n,m}(c_1(x)-c_2(x)) \in D_\rho^{(2)}(\lambda,x)\end{array}\right\}\right.\]
since in this case $\rho_{n,m}(t-c_2(x)) = \rho_{n,m}(c_1(x)-c_2(x))$. If a conditon of type (\ref{harder}) holds on $A_{i,j}^{(\ref{vw1})}$, then there exists some $r$ with $0\leqslant rn<m$ such that
\[\rho_{n,m}(t-c_2) = \rho_{n,m}(c_1-c_2) +_r \rho_{n,m}(t-c_1).\] If we define $B_1$ to be the set
\[B_1 :=\left\{(\lambda, x, \rho) \in D_1\times \Lambda_{n,m} \ \left| \ \begin{array}{l} \rho \in D_{\rho}^{(1)}(\lambda,x),  \\\text{and }  \rho_{n,m}(c_1(x)-c_2(x))+_{r} \rho \in D_{\rho}^{(2)}(\lambda,x)\end{array}\right\}\right.\]  
 then  $A_{i,j}^{(\ref{vw1})}\cap Q$ is equal to the cell
\[A_{i,j}^{(\ref{vw1})}\cap Q= \{(\lambda,x,t) \in A_{i,j}^{(\ref{vw1})} \ | \ \rho_{n,m}(t-c_1) \in B_1(\lambda,x)\}.\]
The situation is completely similar for sets $A_j^{(\ref{vw2})}\cap Q$.\end{proof}
\noindent Our aim is to use cells to give a simple description of sets definable in $\LaffinfQE$ without $K$-quantifiers.  For this we will need the following lemma.

\begin{lemma}\label{lemma:polycenters}
Let $f_1(x,t), \ldots, f_r(x,t)$ be $(\PP_K(\pi),K)$-linear polynomials in variables $(x_1,\ldots, x_n,t)$. There exists a finite partition of $K^{k+1}$ into cells,  $(\PP_K(\pi),K)$-linear polynomials $c(x), d_i(x), h_i(x)$, $a_i \in \PP_K(\pi)$ and a $\Lambda_{n,m}$-polynomial $g_i$ in $r$ variables, such that the following holds for all $f_i(x,t)$ on each cell $A$ with center $c(x)$:
 \begin{enumerate}
\item $\rho_{n,m}(f_i(x,t)) = g_i(\rho_{n,m}(t-c(x)),\rho_{n,m}(d_2(x)), \ldots, \rho_{n,m}(d_r(x))),$
\item  $\ord f_i(x,t) = \left\{\begin{array}{l}\ord h_i(x) \ \text{ for all } (x,t) \in A ,\\ \text{\quad or} \\  \ord a_i(t-c(x)) \ \text{ for all } (x,t) \in A.\end{array}\right.$ \\
\end{enumerate}
\end{lemma}
\begin{proof}
If $r$=1, our claim is trivial, since we can write (if $b\neq 0$):
\[f_1(x,t) = \sum_{i=1}^n a_ix_i +bt+d = b\left(t + \sum_{i=1}^n\frac{a_i}{b} x_i + \frac{d}{b}\right)\]
Now suppose the lemma is true for polynomials $f_1(x,t),\ \ldots, f_{r-1}(x,t)$. This means that there exists a partition of
$K^{k+1}$ in cells $A$ with center $c(x)$, such that on each
cell,
\[\ord f_i(x,t) = \ord a_i(t-c(x)) \text{ or } \ord f_i(x,t) = \ord h_i(x).\] We may assume that $f_r(x,t)=a_r(t-c_r(x))$, for some $(P_K(\pi),K)$-linear polynomial $c_r(x)$. 
Partition $K^{k+1}$ in the following way:
\begin{eqnarray*}K^{k+1} &=& \{(x,t) \in K^{k+1}\ |\
\ord(t-c(x)) < \ord (c(x)-c_r(x)) +m\}\nonumber\\
&  & \cup \ \{(x,t) \in K^{k+1}\ |\
\ord(t-c(x)) > \ord (c(x)-c_r(x)) +m\}\label{partqp}\\
& & \cup\hspace{-3pt} \bigcup_{l=-m}^m\hspace{-5pt} \{(x,t) \in \Q_p^{k+1}\ |\ \ord(t-c(x)) = \ord
(c(x)-c_r(x))+l \}\nonumber
\end{eqnarray*}
Take intersections of the cells $A$ with the above parts of
$K^{k+1}$. By Proposition \ref{prop:intersection}, this results in a
finite partition of $K^{k+1}$ in cells $B$.\\
On each cell $B$, we can now eliminate one of the centers ($c(x)$ or $c_r(x)$). 
For example, if for some $-m\leqslant l<0$, the relation $\ord(t-c(x)) = \ord
(c(x)-c_r(x))+l$ holds on $B$, there exists $r$ with $ 0\leqslant rn <m$ such that
\[\rho_{n,m}(t-c_r(x))= \rho_{n,m}(t-c(x)) +_r \rho_{n,m}(c(x)-c_r(x)), \] so that we can eliminate the center $c_r(x)$ from the description of all polynomials for $(x,t) \in B$. The other cases are similar. 
\end{proof}
\noindent We can now give a characterization of  the subsets of $K^{k+1}$ that are quantifier-free definable in $\LaffinfQE$.

\begin{theorem}
Let $B \subseteq K^{k+1}$ be a set  that is $\LaffinfQE$ - definable without using $K$-quantifiers. There exist $r\in \N,\ n,m \in \N\backslash\{0\}$ and a finite number of disjoint cells $C_i \subseteq \Lambda_{n,m}^r\times K^{k+1}$ such that 
\[B =\{(x,t)\in K^{k+1} \ | \ \exists \lambda \in \Lambda_{n,m}^r  :  (\lambda,x,t) \in \cup_i C_i\}. \]
\end{theorem}
\begin{proof}
By the discussion in Section \ref{subsec:form_of_kqf-formula}, it suffices to show that a set of the following form can be partitioned as a finite union of cells: 
\[ E:=\left\{(\lambda, x, t) \in D_{n,m} \times K^{k+1} \ \left| \ (x,t) \in D \wedge \left(\bigwedge_i \rho_{n,m}(f_i(x,t)) = \lambda_i\right) \right\}\right.\]where $\lambda = (\lambda_1, \ldots, \lambda_r)$. $D$ is a quantifier-free definable subset of $K^{k+1}$ (using only the language of the main sort $K$), and $D_{n,m}$ is a definable subset of $\Lambda_{n,m}^r$ (using the language on the $\Lambda_{n,m}$, and possibly using quantifiers over $\Lambda_{n,m}$.)\\\\
We may suppose that $D$ consists of all $(x,t)\in K^{k+1}$ that satisfy a finite number of relations of the form
\begin{equation} \ord f_{i,1}(x,t) < \ord f_{i,2}(x,t),\label{eq:kqf=cells1}\end{equation}
where the $f_{i,j}(x,t)$ are $(\PP_K(\pi),K)$-linear polynomials. Using Lemma \ref{lemma:polycenters}, we can find a partition of $\Lambda^r\times K^{k+1}$ in cells $A$ with center $c(x)$, such that the residue and order of all polynomials $f_i(x,t)$ and $f_{i,j}(x,t)$ can be expressed as in the formulation of Lemma \ref{lemma:polycenters}. This implies that on $E \cap A$, a relation of the form (\ref{eq:kqf=cells1}) simplifies to either
\begin{equation} \ord (t-c(x)) < \ord h_i(x), \quad \text{or possibly}\quad \ord h_{i,1}(x) < \ord h_{i,j}(x),\label{eq:kqf=cells2}\end{equation}
for some $(\Q(\pi),K)$-linear polynomials $h_i(x), h_{i,j}(x)$. Also, on $E\cap A$, the condition $\bigwedge_i \rho_{n,m}(f_i(x,t)) = \lambda_i$ is equivalent to a formula of the form (for ease of notation, we assume that the center of 
$A$ is the center of $f_1(x,t)$):
\begin{equation}\rho_{n,m}(t-c(x)) = a\lambda_1 \wedge \bigwedge_{i=2}^{r} [\lambda_i = g_i(\lambda_1, \rho_{n,m}(d_2(x)),\ldots, \rho_{n,m}(d_r(x)))],\label{eq:kqf=cells3}\end{equation} for some constant $a \in \PP_K(\pi)$. But this implies that $E \cap A$ is equal to the intersection of $A$ with the cell described by (\ref{eq:kqf=cells2}) and (\ref{eq:kqf=cells3}). By Proposition \ref{prop:intersection}, this can be written as a finite union of cells.
\end{proof}

\subsection{Definable sets and functions} \label{subsec:defsetfun}
Define a semi-additive set to be a set of the following type.
\begin{definition}\label{def:semi-additive}
A set $A \subseteq K^{k+1}$ is called semi-additive if there exist $r\in \N$ and a finite number of disjoint cells $C_i\subseteq \Lambda_{n,m}^r\times K^{k+1}$ such that 
\[A =\{(x,t)\in K^{k+1} \ | \ \exists \lambda \in \Lambda_{n,m}^r  :  (\lambda,x,t) \in \cup_i C_i\}. \]
\end{definition}
\noindent By the next theorem, the semi-additive subsets of $K^{k+1}$ are precisely the $\LaffinfQE$-definable subsets of $K^{k+1}$. And consequently, the $\Laffinf$-definable subsets of $K^k$ are just the semi-additive subsets.
\begin{theorem}
Let $A\subseteq K^{k+1}$ be a semi-additive set. The projection \[B=\{x \in K^k \ | \ \exists t\in K:(x,t) \in A\}\] is a semi-additive set. 
\end{theorem}

\begin{proof}
First, partition the cells $C_i$ occuring in the description of A in smaller cells $C_{i,l}$ such that the extra condition $\ord \rho_{n,m}(t-c(x)) \equiv l \mod n$ holds on $C_{i,l}$. It is then sufficient to prove that we can eliminate the variable $t$ from a formula of the form
\[(\exists t)(\exists \lambda \in \Lambda_{n,m}^r)\left[\begin{array}{cl}& \ord a_1(x) <  \ord (t-c(x)) < \ord a_2(x)\ \\\wedge& \rho_{n,m}(t-c(x))\in D(\lambda,x)\ \\\wedge &\ord \rho_{n,m}(t-c(x))\equiv l \mod n\end{array}\right]\]
and this is equivalent to $(\exists \lambda \in \Lambda_{n,m}^r)\phi(x,\lambda)$, with
\[\phi(x,\lambda) \leftrightarrow (\exists \gamma \in \Gamma_K)\left[\begin{array}{cl}&\ord a_1(x) < \ \gamma < \ord a_2(x) \\ \wedge& \left[D'(\lambda,x)\neq \emptyset\right]\ \wedge\ \left[\gamma \equiv l \mod n\right]\end{array} \right]\]
where $D'(\lambda,x) = D(\lambda,x) \cap \{\mu \in \Lambda_{n,m} \mid \ord \mu \equiv l \mod n\}.$
The formula $\phi(x,\lambda)$ is equivalent with $D'(\lambda,x) \neq \emptyset \wedge (\exists \gamma' \in \Gamma_K)\psi(x, \gamma')$,\ where
\[\psi(x, \gamma') \leftrightarrow \left[ \frac{\ord(a_1(x)\pi^{-l})}{n} < \gamma' <\frac{\ord(a_2(x)\pi^{-l})}{n}\right]\]
Now if $\ord a_1(x)\pi^{-l} \equiv \zeta
\mod n$, for $0 \leqslant \zeta <n$,  then $(\exists \gamma' \in \Gamma_K)\psi(x)$
 is equivalent with
\[ \ord a_1(x)\pi^{-l} + n - \zeta <
\ord a_2(x)\pi^{-l}.\] This completes
the proof, since $\ord a_1(x)\pi^{-l}
\equiv \zeta \mod n$ is a ($K$-quantifier free) $\LaffinfQE$- definable condition on $x$. \end{proof}

\noindent It is now easy to give a characterization of semi-additive functions:

\begin{lemma}
Let $f: B\subseteq K^k \to K^l$ be an $\Laff^{\pi}$-definable function. There exists a finite partition of $B$ in cells $A$ such that $f_{|A}$ has the form
\[f_{|A}: A\to K^l: x\mapsto (f_1(x), \ldots, f_l(x)),\]
where the $f_i(x)$  are $(\PP_K(\pi), K)$-linear polynomials.
\end{lemma}
\begin{proof}
The graph of a definable function is a semi-additive set, so the graph of $f$ can be partitioned as in Definition \ref{def:semi-additive}, using a finite number of cells $C_i$. The fact that $f$ is a function, implies that  for each cell $C_i$, and any $x \in D_K$, there exists a unique $t \in K$ such that $(x,t) \in \mathrm{Graph}(f)$. Note however, that this uniqueness condition implies that $t=c(x)$ and thus the function $f$, when restricted to $D_K$, simply maps each $x$ to the center $c(x)$ of the corresponding cell $C_i$, which we assumed to be a $(\PP_K(\pi),K)$-linear polynomial.
\end{proof}

\section{The case of a finite residue field} \label{sec:finres}

For the following class of fields, angular component maps can be defined in a unique way. Note that we do not require the valued field to be Henselian.

\begin{definition}
Let $\FF_q$ be the finite field with $q$ elements and $\ZZ$ the ordered abelian group of integers. We define a  $(\FF_q,\ZZ)$-field  to be a valued field with residue field
isomorphic to $\FF_q$ and value group elementary equivalent to $\ZZ$.
\end{definition}
\noindent  Fix an $(\FF_q,\ZZ)$-field  $K$, fix an element $\pi$ with smallest positive order, such that $\ord \pi =1$. For each integer $n>0$, let $P_n$ be the set of nonzero $n$-th powers in $K$. 
\begin{lemma}\label{lemma:acm} For each integer $m>0$, there is a unique group homomorphism
$$
\acm:K^\times \to (R_K\bmod \pi^m)^\times
$$
such that $\acm(\pi)=1$ and such that $\acm(u)\equiv u \bmod \pi^m$ for any unit $u\in R_K$.
\end{lemma}

\begin{proof}
Put $N_m:=(q-1)q^{m-1}$ and let $U$ be the set $P_{N_m}\cdot R_K^\times$.
Note that $K^\times$ equals the finite disjoint union of the sets $\pi^\ell\cdot U$ for integers $\ell$ with $0 \leqslant \ell \leqslant N_m-1$.
Hence, any element $y$ of $K^\times$ can be written as a product of the form $\pi^\ell x^{N_m} u$, with $u\in R_K^\times$, $\ell\in\{0,\ldots,N_{m}-1\}$, and $x\in K^\times$.

Since $\acm$ is required to be a group homomorphism to a finite group with $N_m$ elements, it must send $P_{N_m}$ to $1$. Also note that the projection $R_K \to R_K\bmod \pi^m$ (which is a ring homomorphism), induces a natural group homomorphism $p: R_K^\times \to (R_K\bmod \pi^m)^\times$. Now if we write $y=\pi^\ell x^{N_m} u$, we see that $\acm$ must satisfy
\begin{equation}\label{acm}
\acm ( y ) = p(u),
\end{equation}
which implies that the map $\acm$ is uniquely determined if it exists. Moreover, we claim that we can use \eqref{acm} to define $\acm$. This is certainly a well defined group homomorphism:
if one writes $y=\pi^\ell \tilde{ x}^{N_m} \tilde{u}$ for some other $\tilde{u}\in R_K^\times$ and $\tilde{x}\in K^\times$, then clearly $p(u)=p(\tilde {u})$. It is also clear that this homomorphism  sends $\pi$ to $1$ and  satisfies our requirement that $\acm(u)\equiv u \bmod \pi^m$ for any unit $u\in R_K$. 
\end{proof}
$(\F_q,\Z)$-fields satisfy the requirements we listed in the introduction, so if we consider the structure induced by our multi-sorted language, we can apply the cell decomposition results from the previous section. 
 Obviously, since the residue field is now assumed to be finite,  $\Lambda_{n,m}$ will be a finite set. In fact, we can assume that $\Lambda_{n,m}$ is a subset of $R_K$, by choosing a fixed set of representatives for each equivalence class. For example, if $K = \Q_p$, we could take
 \[ \Lambda_{n,m}:= \{p^ra \mid 0\leqslant r < n \wedge \ord a = 0 \wedge 0 < a \leqslant p^{m}-1\}.\] The fact that $\Lambda_{n,m}$ is finite implies that all $\Lambda_{n,m}$-quantifiers can be replaced by conjunctions (for $\forall$) and disjunctions (for $\exists$) over the elements of $\Lambda_{n,m}$.    In particular, 
 if we 
 consider the  2-variable relation
\[S_{n,m}(x,z) \leftrightarrow \rho_{n,m}(x) = \rho_{n,m}(z)),\]
it is possible to `collapse' $\LaffinfQE$ to a mono-sorted language  $ \Laff := (+,-,\cdot_{\pi}, |, \{S_{n,m}\}_{n,m}).$

 It follows immediately from the results of the previous section that
every definable set in this new language is a finite union of cells of the form
\begin{equation} \{(x,t) \in D\times K \ | \ \ord a_1(x) \ \square_1 \ \ord (t-c(x)) \ \square_2 \ \ord a_2(x)\ \wedge \ \rho_{n,m}(t-c(x)) = \lambda\}, \label{eq:collapsedcell}\end{equation}
with $D$ a quantifierfree definable subset of $K^k, \lambda \in \Lambda_{n,m}$; $a_i(x)$ and $c(x)$ are $(\PP_K(\pi),K)$-linear polynomials, and $\PP_K$ is the prime subfield of $K$.

We should compare this with the semi-linear language $(+,- \{\overline{c}\}_{c\in \Q_p}, \{P_n\}_{n\in \N})$ that Liu \cite{liu-94} considered for $\Q_p$. 
A first difference is the use of the relation $S_{n,m}$, instead of the sets of $n$-th powers $P_n$. This difference is much smaller than it may seem at first.
If we define $Q_{n,m}$ to be the set \[Q_{n,m}:=\{x\in K \ | \ \rho_{n,m}(x) = \rho_{n,m}(1)\}\]  then the relation $S_{n,m}(x,y)$ is equivalent to $ x \in yQ_{n,m}$. 
Hence, we replaced expressions like `$x$ is in some coset of $P_n$' by similar expressions that use
 sets $Q_{n,m}$ instead. 
 However, for Henselian $(\F_q,\Z)$-fields, it is easy to see that for any $N \in \N$, $P_N$ can be defined as a finite union of cosets $\lambda Q_{n,m}$ with $\lambda \in K; n,m \in \N$. Since we used cosets of $P_N$ to define the maps $\acm$ (and thus the sets $Q_{n,m}$), the converse is also true. 

Another (seeming) difference is that the language we defined contains the divisibility symbol `$|$'. Liu does not include this symbol, since he showed that for semi-linear sets over $\Q_p$, this relation is quantifierfree definable. We need the symbol if we want to achieve quantifier elimination, but it can be shown, see \cite[Proposition 1]{clu-lee-2011}, that for $(\F_q, \Z)$-fields, the relation $\ord(x-z) < \ord (y-z)$ is definable whenever the relation $\rho_{n,m}(y-x) = \rho_{n,m}(z)$ is definable. So adding the symbol to our language does not affect the number of definable sets. 

A third difference lies in the amount of scalar multiplication which is definable. $\Laff$ has less scalar multiplication than the semi-linear language. (To compare: for the structure $(\Q_p; +,-,\cdot_{\pi}, \{S_{n,m}\}_{n,m})$, scalar multiplication is only definable for constants from $\Q$.) This difference will be important when we compare the definable functions. 
\\\\
Taking these observations into account, we can consider a class of  \emph{semi-affine} structures:
\begin{definition}
Given an $(\F_q,\Z)$-field $K$ and a subfield $L \subseteq K$, let
 $\Laff^L$ be the language
\[ \Laff^L := (+,-, \{\overline{c}\}_{c\in L}, |, \{R_{n,m}\}_{n,m}).\]
The structure $(K, \Laff^L)$ is called a semi-affine structure.
\end{definition}
These languages are variations on the language $\Laff$ we defined above, adding additional symbols for scalar multiplication, and replacing the symbol $S_{n,m}$ by $R_{n,m}$,  a relation which is defined as $R_{n,m}(x,y,z) \leftrightarrow \rho_{n,m}(y-x) = \rho_{n,m}(z)$. We make this (otherwise unnecessary) substitution to point out the link with the language $(\{R_{n,m}\}_{n,m})$, that we studied in \cite{clu-lee-2011}.  

Over $\Q_p$,  Liu's semi-linear language is equivalent with $\Laff^{\Q_p}$. Note also that the structures $(K,\Laff^{\PP_K(\pi)})$ and $(K, \Laff)$ have the same definable sets.
In general, when considering a structue $(K, \Lm)$, we will always assume that if multiplication by $c$ is definable, then $\Lm$ contains a symbol $\overline{c}$ (replacing $\Lm$ by a definitional expansion if necessary). In particular, we assume that $\PP_K \subseteq L$.

To describe the definable sets and functions of such structures, the following terminology is useful.
\begin{definition} Let $L \subseteq K$ be fields.
An $(L,K)$-linear polynomial is a polynomial of the form
\[ a_1x_1 + \ldots + a_nx_n +b,
\qquad\text{with } a_i \in L \text{ and } b \in K.\]
If $\Lm = \Laff^L$, we write $\text{Poly}_k(\Lm,K)$ for the set of all $(L,K)$-linear polynomials in $k$ variables.
\end{definition}
For all semi-affine structures $(K,L)$, we can deduce cell decomposition and quantifier elimination, using the method we described for the language $\Laff$. The general idea is this: the cell decomposition results from the previous section still hold if we consider variations of $\LaffinfQE$, where we have more (or less) symbols for scalar multiplication to the language of the field sort. Every semi-affine language can be obtained by collapsing such a language to a language having only the field sort. In each case, we obtain cell decomposition using cells as in \eqref{eq:collapsedcell}, where the only difference is that for $(K,\Laff^L)$, the functions $a_i(x)$ and $c(x)$ will now be $(L,K)$-linear polynomials. (Assuming that scalar multiplication is only definable for constants from $L$.) From this, the following description of definable cells can easily be deduced:
\begin{lemma}
The definable sets of a semi-affine structure $(K, \Lm)$ are the boolean combinations of sets of the forms
\[\{ x\in K^k \ | \ \ord f_1(x) \ \square\ \ord \pi^r f_2(x) \} \mathand \{x\in K^k \ | \ \rho_{n,m}(f_3(x)) = \lambda \},\]
where the  $f_i \in \text{Poly}_k(\Lm,K)$, $r\in \Z$ and $\lambda \in \Lambda_{n,m}$. 
\end{lemma}

\noindent In the next section we study the definable functions for these languages.

\subsection{Definable functions and Skolem functions} \label{subsec:skol}

The definable functions of a semi-affine structure $(K, \Lm)$ will be called $\Lm$-semiaffine functions over $K$. The definable sets and functions of $(\Q_p, \Laff^{\Q_p})$ will be referred to as being `semi-linear'.
 Using cell decomposition, it is easy to see that semi-affine functions
actually have a very simple form.

\begin{lemma} 
Let $(K, \Lm)$ be a semi-affine structure. For any $\Lm$-semiaffine function $A \subseteq K^k \to K^l$ there exists a finite partition of $A$ in $\Lm$-definable sets $B_i$, such that $f_{|B_i}$ has the form
\[f_{|B_i}: B_i \to K^l : x\mapsto (f_1(x), \ldots, f_l(x)),\]
with $f_l(x) \in \text{Poly}_k(\Lm,K)$.
\end{lemma}
All of these semi-affine structures are truly linear in the sense that there does not exist any open set where multiplication is definable.

\begin{corol} Let $K$ be any $(\F_q, \Z)$-field and $\Lm$ a semi-affine language. 
Let $U \subseteq K^2$ be an open semi-affine set. The map $f:U \to K:(x,y) \mapsto xy$ is not a semi-affine function.
\end{corol}
\begin{proof}
Let us assume that scalar multiplication is definable for all $c \in K$, and that 
multiplication is definable on an open cell $C$. 
Fix a point $(x_0,y_0) \in C$. It is easy to see that if we choose $k \in \N$ big enough, we have that 
\begin{equation}\label{eq:nomult} \{ (x,y) \in K^2 \mid x \in x_0 + \pi^kR_K, y\in  y_0 + \pi^kR_k\} \subset C.\end{equation} If $\ord x_0 \leqslant \ord y_0$, there exists $\alpha \in R_K$ such that $y_0 = \alpha x_0$. Moreover, because of \eqref{eq:nomult}, the intersection \[W := C\cap \{(x,y) \in K^2 \mid y = \alpha x\}\] is an infinite set, and  the projection $\pi_x(W)$  onto the first coordinate also contains infinitely many points.
Note that since $xy = \alpha x^2$ for $(x,y) \in W$, the multiplication map on $W$ induces a definable function $\pi_x(W) \to K: x \mapsto \alpha x^2$.
\\
After some (finite) further partitioning, we can find an open subset $U \subseteq \pi_x(W)$ and constants $b_1, b_2$ such that on $U$, the function $f(x) = b_1 x + b_2$ defines the map $x \mapsto \alpha x^2$. But this implies that the equation $b_1x + b_2 = \alpha x^2$ has infinitely many solutions, which is a contradiction.  If $\ord x_0 > \ord y_0$, we can give a similar argument by intersecting with the set $\{x =\frac1 \alpha y\}$.
\end{proof}

\noindent A question one can pose concerning semi-affine functions is whether it is always possible to find a definable Skolem
function, i.e. a definable choice in the fibers of $f$. As is the case for semi-algebraic functions (see \cite{scow-vdd88}), the answer is certainly `yes' for semilinear functions, and more generally, for functions definable in a structure $(K,\Laff^K)$.

\begin{theorem}
Let $X \subseteq K^{k + r}$ be an $\Laff^K$-definable set. \\If $\pi_k(X) \subseteq K^k$  
is the projection on the first $k$ variables,  there exists a semilinear function $g: \pi_k(X) \to X$ 
such that $\pi_k \circ g = \mathrm{Id}_{\pi_k(X)}$.
\end{theorem}
\begin{proof}
If suffices to check that 
given a $C$ and the projection map $\pi_x$,
\[\pi_x: C \subset K^{l+1} \to K^l: (x_1,\ldots, x_n,t) \mapsto (x_1, \ldots,
x_n),\] there exists a definable function $g: \pi_x(C) \to C$ such
that $\pi_x \circ g = \mathrm{Id}_{\pi_x(C)}$.
\\
If the cell $C$ has a center $c(x) \neq 0$, we first apply a
translation
\[C\to C': (x,t) \mapsto (x,t-c(x)),\]
to a cell $C'$ with center $c'(x)=0$. Since this translation is
bijective, it is invertible. Therefore the problem is reduced to the following. 
Let $C$ be a cell of the form
\[C = \{(x,t) \in D \times K \ | \ \ord b(x)\,\square_1 \,\ord
t \,\square_2 \,\ord a(x)\ \wedge \ \rho_{n,m}(t) = \lambda \},
\] where $a(x), b(x) \in K[x]$  and
$D \subseteq K^l$ is a definable set. 
We must show that there exists a definable function $g: \pi_x(C) \to C$ such
that $\pi_x \circ g = \mathrm{Id}_{\pi_x(C)}$. \\\\
Given $x \in \pi_x(C)\subseteq D$, we have to find $t(x)$ such that
$(x,t(x))$ satisfies the conditions
\begin{eqnarray}
 \ord b(x)\ \square_1 \ \ord
t(x) \ \square_2 \ \ord a(x) \label{c1}\\
\rho_{n,m}(t(x)) =  \lambda \label{c2}
\end{eqnarray}
If $\lambda = 0$, put $g(x) = (x,0)$. From now on we assume that $\lambda \neq 0$.\\
If $\square_1 = \square_2 =$ `no condition', we can simply put $g(x)
=(x, \lambda).$\\
If $\square_2 = $ `$<$', we can define $g$ as follows. First partition
$\pi_x(C)$ in parts $D_{\mu}$, such that
\[ D_{\mu} = \{x\in \pi_x(C) \ | \ \rho_{n,m}(a(x)) = \mu \}.\]
(Note: if $\mu = 0$, we can reduce to the cases were $\square_2$ = `no condition'.)
Our strategy is based on the fact that for every $x \in D$, there
exists $k \in \Z$ such that $k$ satisfies
\[\ord b(x) \ \square_1 \ \ord \lambda + kn < \ord a(x). \]
Restricting to a set $D_{\mu}$, we construct an element $t(x)$
with order as close as possible to $\ord a(x)$. This ensures that
$t(x)$ satisfies (\ref{c1}).
The definiton of $g$ on $D_{\mu}$ will depend on
the respective orders of $\lambda$ and $\mu$.
\begin{itemize}
\item If $\ord \lambda < \ord \mu$, we can define $g_{|D_{\mu}}$ as
\(g_{|D_{\mu}}: D_{\mu} \to C: x \mapsto \left(x,\frac{\lambda}{\mu} a(x)\right).\)
This means that we put $t(x) = \frac{\lambda}{\mu} a(x)$. Clearly
 $\rho_{n,m}(t(x)) = \lambda$. Also, since $-n < \ord
(\frac{\lambda}{\mu}) < 0$, we have that $0<\ord
\frac{a(x)}{t(x)}<n$, and thus condition (\ref{c1}) must be
satisfied.
\item If $\ord \lambda \geqslant \ord \mu$, put
\(g_{D_{\mu}}: D_{\mu} \to C: x \mapsto \left(x,\frac{\lambda}{\pi^n\mu} a(x)\right).\)
\end{itemize}
If  $\square_1 =$ `$<$' and $\square_2 =$ `no condition', we choose
$t(x)$ with order as close as possible to $\ord b(x)$. More
specifically, if $\ord \lambda \leqslant \mu$, define $g$ as
\(g_{D_{\mu}}: D_{\mu} \to C: x \mapsto \left(x,\frac{\lambda \pi^n}{\mu}
b(x)\right)\), and if $\ord \lambda > \ord \mu,$ put
\(g_{D_{\mu}}: D_{\mu}\to C: x \mapsto \left(x,\frac{\lambda}{\mu} b(x)\right).\)
\end{proof}
\noindent One has to be more careful for  structures $(K,\Laff^L)$ where $L \neq K$: the following lemma gives an example of a semi-affine structure
that has no definable Skolem functions.

\begin{lemma}\label{lemma:skolemscalars}
Let $K$ be an $(\F_q, \Z)$- field (with $q = p^r$\hspace{-2pt}) such that $\mathrm{char}(K) = 0$, and suppose that $\ord \pi < \ord p$. Let $A$ be the set
 \[A:= \{(x,y) \in K^2 \ | \ \ord y = 1+\ord x\}.\]
 For the projection map $\pi_1: A \to K: (x,y)\mapsto x,$
there does not exist an $\Laff^{\Q}$-definable function $g: \pi_1(A) \to K^2$ such that $\pi_1 \circ g = \mathrm{Id}_{\pi_1(A)}$.
\end{lemma}
\begin{proof}
Suppose such a $g$ exists.
After partitioning  $\pi_1(A)$ in cells $C$, the function $g$ must have the form
\[g_{|C}: C\to K^2 : x \mapsto (x, ax+b),\]
where $ax+b$ is a $(\Q, K)$-linear polynomial, and hence $a \in \Q$.
There must be at least one cell $C$ that contains elements $x$ for which $\ord x < \ord \frac{b}{a}$. For these elements, $\ord ax+b = \ord ax$. However, since $\ord p > \ord \pi=1$ and $\ord a \in (\ord p)\Z$, it is impossible that $\ord a =1$, which is a contradiction.
\end{proof}
In general, $(K, \Laff^L)$ will admit definable skolem functions if for any $n,m \in \N$ and for any coset $\lambda Q_{n,m}$, there exists a element $\lambda_0\in K$ with $\rho_{n,m}(\lambda_0) = \lambda$ and $0 \leqslant \ord \lambda_0<n$ such that scalar multiplication by $\lambda_0$ is definable. This condition is satisfied for $p$-adically closed fields if we require that $\overline{\Q}_K \subseteq L$, where $\overline{\Q}_K$ is the algebraic closure of $\Q$ in $K$.


\subsection{Classification} \label{subsec:clas}

Write $q_K$ for the cardinality of the residue field of $K$. Let $|\cdot|$ be the norm defined as $|x| = \max(|x_i|_K)$, where $|x_i|_K = q_K^{-\ord(x_i)}$. We can define a dimension invariant for semi-affine structures by using the notion of dimension that Scowcroft and van den Dries \cite{scow-vdd88} introduced for semi-algebraic sets, i.e., \emph{the dimension of a definable set $X$ is the greatest natural number $n$ such that there exists a non-empty definable subset $A \subseteq X$ and a definable bijection from $A$ to a nonempty definable open subset of $K^n$.}
It is straightforward, using cell decomposition and our characterization of definable functions, to check that this notion of dimension has the expected properties when applied to the context of semi-affine sets.
\\\\
Cluckers \cite{clu-2000} showed that two infinite $p$-adic semi-algebraic sets are isomorphic (i.e. there exists a definable bijection) if and only if they have the same dimension. There exists no analogous result for the semi-affine case, however. We will illustrate this fact with some examples. Although most results presented below are true for all $(\F_q, \Z)$-fields, we will restrict our attention to $K=\Q_p$.
\begin{lemma}
There exists no semi-affine bijection between the sets
\[ A = \{ t\in \Q_p \ | \ord t <0 \}\qquad \text{ and}  \qquad B =\{ t \in \Q_p \ | \ord t >0\}.\]
\end{lemma}
\begin{proof}
Suppose such a bijection $f: A \to B$ exists. Then there must
exist a finite partition of $A$ in sets $A_i$ such that $f$ is
linear on each $A_i$. Since this partition is finite, at least one
of these sets $A_i$ must contain a subset of the form \[C_i=\{t\in \Q_p \ |\ \ord t <
-k \wedge \rho_{n,m}(t) = \lambda\},\] with $k \in \N$. 
By our assumption, there
must exist $a\in\Q$ and $b \in \Q_p$ such that on $A_i$, the map
$f_{|A_i}$ has the form $f_{|A_i}: x \mapsto ax+b$. If $f$ is
indeed a bijection, then $f(C_i)$ must be a subset of $B$, and
thus the condition $\ord f(x)
>0$ has to hold for all $x \in C_i$. However,
 it is possible to take $x\in C_i$ such that $\ord x <
\min\{\ord\hspace{-3pt}\left(\frac{b}{a}\right),
\ord\hspace{-3pt}\left(\frac1a\right)\}$. But then $\ord f(x) =
\ord (ax+b)<0$.
\end{proof}

\noindent Other examples of non-isomorphic sets of the same dimension are the sets $P_n$. 
To obtain this result, we will first look at the sets $Q_{n,m}$.
For most pairs $(n,m)$, the sets $Q_{n,m}$ are
essentially different. More precisely, there exists an isomorphism
between $Q_{n,m}$ and $Q_{n',m'}$ if and only if $n' = np^{m-m'}$.
To prove this, we first need the following lemma. (Note: We use
the notation $A\sqcup B$ to denote the disjoint union of two sets
$A$ and $B$. In practice this can be defined as $\{0\} \times A \cup \{1\} \times B$.)

\begin{lemma} \label{lemma:nobijection}
There exists no semi-affine bijection between
\[ \bigsqcup_{i\in I_1} Q_{n,m}
\qquad \text{and} \qquad  \bigsqcup_{i \in I_2} Q_{n,m} \] if
$I_1$ and $I_2$ are index sets with different cardinalities.
\end{lemma}
\begin{proof}
For $j \in I_1$, we denote the different copies of $Q_{n,m}$ by
$Q_{n,m}^{(j)}$. Suppose a semi-affine bijection
\[f:\bigsqcup_{i\in I_1} Q_{n,m}
\to  \bigsqcup_{i \in I_2} Q_{n,m}\] does exist. Then there must
exist a finite partition of the $Q_{n,m}^{(j)}$ in cells $C$ such
that $f_{|C}$ is linear. Since we take finite partitions, for each
$Q_{n,m}^{(j)}$, there must be at least one cell of the form $ \{
x \in \Q_p \ | \ \ord x < k , x \in \lambda_{ij}Q_{n_{ij},m_{ij}}
\}$. In fact, after a further finite partition, we may suppose
that $n_{ij}$ and $m_{ij}$ are equal for each cell, and thus that
all $x \in \bigsqcup_{i\in I_1} Q_{n,m}$ with order smaller than
some fixed integer $k$ belong to a set in the partition which has
the form
\[C_{k, \lambda}^{(j)}:= \{ x \in \Q_p \ | \ \ord x < k , x \in
\lambda Q_{n',m'} \}.\] Because of the previous lemma, we will
have to map the elements of these cells to the elements with very
small (negative) order of $\bigsqcup_{i \in I_2} Q_{n,m}$ to get a
bijection. \\
It is easy to see that if $k < \ord(\frac{b}{a})-m'$,  a
function $x \mapsto ax+b$ gives a bijection between $C_{k, \lambda}^{(j)}$ to $C_{k+\ord a,
a\lambda}^{(j')}$. 
\\
 If we choose $k \in \Z$ small enough, then every set $C_{k,\lambda}^{(j)} \subset \bigsqcup_{j \in I_1} Q_{n,m}^{(j)}$ is mapped to a set $C_{k',\lambda}^{(j')} \subset \bigsqcup_{j' \in I_2} Q_{n,m}^{(j')}$. Also, for small enough $k' \in \Z$, every set $C_{k',\lambda}^{(j')}$ is in the image of exactly one set $C_{k,\lambda}^{(j)}$. So if $f$ is the required bijection, then for a small enough value of $\ell$,  $\bigsqcup_{j \in I_1} Q_{n,m}^{(j)}$ and $\bigsqcup_{j' \in I_2} Q_{n,m}^{(j')}$ contain exactly the same number of sets of the form $\{ x \in \Q_p \ | \ \ord x < \ell ,\ x \in
\lambda Q_{n',m'} \}$, which is only possible if $I_1$ and $I_2$ have the same cardinality.
\end{proof}

\begin{corol}
There exists a semi-affine bijection between $Q_{n,m}$ and
$Q_{n',m'}$ if and only if $n'= np^{m-m'}$.
\end{corol}
\begin{proof}
Suppose $m = \max\{m,m'\}$ and partition $Q_{n,m}$ and
$Q_{n',m'}$ as
\begin{eqnarray*}
Q_{n,m}= \bigcup_{\lambda \in I_1} \lambda Q_{nn',m} &\text{ \ \ and \ \ } &
Q_{n',m'} = \bigcup_{\lambda \in I_2} \lambda Q_{nn',m}.
\end{eqnarray*}
Here $I_1$ and $I_2$ are defined als follows:
\begin{eqnarray*}
I_1 &=& \{ 1,p^2,\ldots, p^{(n'-1)n}\},\\
I_2 &=&\{p^{rn'}(1+a_{m'}p^{m'}+ \ldots + a_{m-1}p^{m-1}) \ | \
0\leqslant r < n;\  0\leqslant a_i \leqslant p-1\}.
\end{eqnarray*}
 If there exists a semi-affine bijection between $Q_{n,m}$ and
$Q_{n',m'}$, this induces a bijection
\[\bigsqcup_{i\in I_1} Q_{nn',m}
\to  \bigsqcup_{i \in I_2} Q_{nn',m}.\] 
But since $\# I_1= n'$ and $\#I_2 = np^{m-m'}$, this contradicts Lemma \ref{lemma:nobijection} if $n'\neq np^{m-m'}$.\\
If the cardinalities of $I_1$ and $I_2$ are equal, let  $\tau$ be a bijection between $I_1$
and $I_2$. Now put
\[f_{|\lambda Q_{nn',m}}(x) = \frac{\tau(\lambda)}{\lambda}\, x.\] 
The function $f:Q_{n,m}\to Q_{n',m'}$ is the required bijection.
\end{proof}

\begin{corol} Let $n,n' >0$. \\There exists a semi-affine bijection between $P_n$ and $P_{n'}$ if and only if 
\[\frac{\#\Lambda_n}{\#\Lambda_{n'}} = \frac{n}{n'}\,p^{2\ord \left(\frac{n}{n'}\right)},\] 
where $\Lambda_n:= P_n \cap \{x\in \Q \ | 0<x \leqslant p^{2\ord n +1}-1\} $.\\
In particular, if $p \nmid n$ and $p\nmid n'$, there is no semi-affine bijection between $P_n$ and $P_{n'}$ if $n \neq n'$.
\end{corol}
\begin{proof}
Take partitions $P_n = \bigcup_{\lambda \in \Lambda_n} \lambda Q_{n,2\ord n+1},$ (and similarly for $P_{n'}$), as explained before. Assume that $\ord n \geqslant \ord n'$. By a similar reasoning as in the proof of the previous corollary, a bijection between $P_n$ and $P_{n'}$ would induce a bijection \[\bigsqcup_{i\in I_n} Q_{nn',2 \ord n+1}
\to  \bigsqcup_{i \in I_{n'}} Q_{nn',2 \ord n +1},\]
with $\# I_n= n'\cdot \#\Lambda_n$ and $\#I_{n'} =  np^{2(\ord n - \ord n')}\cdot \#\Lambda_{n'}$. There exists a bijection if and only if $\# I_n = \# I_{n'}$.\\\\
Now assume that $\ord n = \ord n' =0$, and $n \geqslant n'$. There exists a bijection between $P_n$ and $P_{n'}$ if $\frac{n}{\#\Lambda_n} = \frac{n'}{\#\Lambda_{n'}}$. Under our assumptions, $\#\Lambda_n$ is equal to the  number of elements of $\F_p^{\times}$ that are $n$-th powers. Applying a result from elementary number theory, we get that  $\#\Lambda_n = \frac{p-1}{d}$, where $d = (p-1,n)$, and therefore $P_n$ will be isomorphic with $P_{n'}$ if and only if $nd = n'd'$. This is equivalent to $n\tilde{d} = n'\tilde{d'}$, with $\tilde{d} = \frac{d}{a}, \tilde{d'}=\frac{d'}{a}$, and $a = (d,d')$. As a consequence, $\tilde{d'} \mid n$. 
If $\tilde{d'} \neq 1$, there is $q >1$ such that $q | \tilde{d'}$. But then also $q \mid (p-1,n)$. This contradicts $(\tilde{d'},d)=1$, so 
we conclude that $\tilde{d'}=1$, and therefore $n\tilde{d} =n'$, which contradicts our assumption that $n \geqslant n'$, unless $\tilde{d}=1$ and $n=n'$.
\end{proof}
\subsection*{Acknowledgements}
The results presented in this paper were obtained as part of my PhD thesis. I would like to thank my supervisor, Raf Cluckers, for many stimulating conversations about this topic, and other members of the jury (in particular, Jan Denef, Angus Macintyre and Leonard Lipshitz) for useful comments. Many thanks also to the Math Department of K.U.Leuven, for providing financial support and a very stimulating working environment. I would also like to thank the referee.
%
%
\bibliographystyle{abbrv}

\bibliography{/Users/iblueberry/Documents/Bibliografie}

\begin{thebibliography}{10}

\bibitem{clu-2000}
R.~Cluckers.
\newblock Classification of semi-algebraic {$p$}-adic sets up to semi-algebraic
  bijection.
\newblock {\em J. Reine Angew. Math.}, 540:105--114, 2001.

\bibitem{clu-lee-2008}
R.~Cluckers and E.~Leenknegt.
\newblock Rectilinearization of semi-algebraic {$p$}-adic sets and {D}enef's
  rationality of {P}oincar{\'e} series.
\newblock {\em J. Number Theory}, 128(7):2185--2197, 2008.

\bibitem{clu-lee-2011}
R.~Cluckers and E.~Leenknegt.
\newblock A version of $p$-adic minimality.
\newblock {\em Journal of Symbolic Logic}, 77(2):621--630, June 2012.

\bibitem{clr-06}
R.~Cluckers, L.~Lipshitz, and Z.~Robinson.
\newblock Analytic cell decomposition and analytic motivic integration.
\newblock {\em Ann. Sci. {\'E}cole Norm. Sup. (4)}, 39(4):535--568, 2006.

\bibitem{clu-loe-07}
R.~Cluckers and F.~Loeser.
\newblock b-minimality.
\newblock {\em J. Math. Log.}, 7(2):195--227, 2007.

\bibitem{denef-86}
J.~Denef.
\newblock {$p$}-adic semi-algebraic sets and cell decomposition.
\newblock {\em J. Reine Angew. Math.}, 369:154--166, 1986.

\bibitem{fle-2011}
J.~Flenner.
\newblock Relative decidability and definability in {H}enselian valued fields.
\newblock {\em J. Symbolic Logic}, 76(4):1240--1260, 2011.

\bibitem{has-mac-97}
D.~Haskell and D.~Macpherson.
\newblock A version of o-minimality for the {$p$}-adics.
\newblock {\em J. Symbolic Logic}, 62(4):1075--1092, 1997.

\bibitem{lee-2011b}
E.~Leenknegt.
\newblock Cell decomposition and definable functions for weak $p$-adic
  structures.
\newblock {\em Submitted}.

\bibitem{lee-2012.2}
E.~Leenknegt.
\newblock Cell decomposition for semi-bounded $p$-adic sets.
\newblock (To be submitted), Feb. 2012.

\bibitem{lee-2012.1}
E.~Leenknegt.
\newblock Reducts of $p$-adically closed fields.
\newblock Preprint, Feb. 2012.

\bibitem{liu-94}
N.~Liu.
\newblock Semilinear cell decomposition.
\newblock {\em J. Symbolic Logic}, 59(1):199--208, 1994.

\bibitem{mac-76}
A.~Macintyre.
\newblock On definable subsets of $p$-adic fields.
\newblock {\em J. Symb. Logic}, 41:605--610, 1976.

\bibitem{mpp-92}
D.~Marker, Y.~Peterzil, and A.~Pillay.
\newblock Additive reducts of real closed fields.
\newblock {\em J. Symbolic Logic}, 57(1):109--117, 1992.

\bibitem{mou-09}
M.-H. Mourgues.
\newblock Cell decomposition for {$P$}-minimal fields.
\newblock {\em MLQ Math. Log. Q.}, 55(5):487--492, 2009.

\bibitem{pas-89}
J.~Pas.
\newblock Uniform {$p$}-adic cell decomposition and local zeta functions.
\newblock {\em J. Reine Angew. Math.}, 399:137--172, 1989.

\bibitem{pas-90}
J.~Pas.
\newblock Cell decomposition and local zeta functions in a tower of unramified
  extensions of a {$p$}-adic field.
\newblock {\em Proc. London Math. Soc. (3)}, 60(1):37--67, 1990.

\bibitem{pet-92}
Y.~Peterzil.
\newblock A structure theorem for semibounded sets in the reals.
\newblock {\em J. Symbolic Logic}, 57(3):779--794, 1992.

\bibitem{pet-93}
Y.~Peterzil.
\newblock Reducts of some structures over the reals.
\newblock {\em J. Symbolic Logic}, 58(3):955--966, 1993.

\bibitem{sca-03}
T.~Scanlon.
\newblock Quantifier elimination for the relative {F}robenius.
\newblock In {\em Valuation theory and its applications, {V}ol. {II}
  ({S}askatoon, {SK}, 1999)}, volume~33 of {\em Fields Inst. Commun.}, pages
  323--352. Amer. Math. Soc., Providence, RI, 2003.

\bibitem{scow-vdd88}
P.~Scowcroft and L.~van~den Dries.
\newblock On the structure of semialgebraic sets over {$p$}-adic fields.
\newblock {\em J. Symbolic Logic}, 53(4):1138--1164, 1988.

\bibitem{vdd-98}
L.~van~den Dries.
\newblock {\em Tame topology and o-minimal structures}, volume 248 of {\em
  London Mathematical Society Lecture Note Series}.
\newblock Cambridge University Press, Cambridge, 1998.

\end{thebibliography}
\end{document}